\documentclass[a4paper,11pt]{article}

\usepackage{stmaryrd}
 
\usepackage[T1]{fontenc}
\usepackage[utf8]{inputenc}
\usepackage{lipsum}
\usepackage{amsmath, amssymb, mathtools, amsthm}
\usepackage{esint}
\usepackage{comment}
\usepackage{esint}
\usepackage[usenames,dvipsnames]{xcolor}
\usepackage{graphicx}
\usepackage{subfigure}
\usepackage{minitoc}
\usepackage{tikz}
\usepackage{enumerate}
\usepackage[margin=1.2 in]{geometry}
\usepackage{bbm} 
\usepackage[numbers,sort&compress]{natbib}
\usepackage[colorlinks=true]{hyperref}
\hypersetup{urlcolor=blue, citecolor=red}
\usepackage{mathtools}

\usepackage[title]{appendix}
\usepackage{dsfont}

\DeclarePairedDelimiter{\abs}{\lvert}{\rvert}
\DeclarePairedDelimiter{\norm}{\lVert}{\rVert}

\DeclarePairedDelimiter{\prt}{(}{)}

\newcommand{\curlyE}{\mathcal{E}}

\newcommand{\curlyI}{\mathcal{I}}

\newcommand \commentout[1] {}

\DeclareMathOperator*{\supp}{\operatorname{supp}}

\newcommand{\partialt}[1]{\dfrac{\partial#1}{\partial t}}

\DeclareMathAlphabet{\mathup}{OT1}{\familydefault}{m}{n}
\newcommand{\dx}[1]{\mathop{}\!\mathup{d} #1}

\theoremstyle{plain}
\newtheorem{theorem}{Theorem}[section]
\newtheorem{lemma}[theorem]{Lemma}
\newtheorem{proposition}[theorem]{Proposition}
\newtheorem{corollary}[theorem]{Corollary}

\theoremstyle{remark}
\newtheorem{remark}[theorem]{\bf Remark}

\newcommand{\ie}{\emph{i.e.}}

\newcommand{\ddt}{\frac{\dx{}}{\dx{t}}}

\newcommand{\grad}{\nabla}
\renewcommand{\div}{\nabla\cdot}

\newcommand{\R}{\mathbb{R}}

\newcommand{\T}{\mathbb{T}}

\begin{document}

\title{Uniform regularity estimates for nonlinear diffusion-advection equations in the hard-congestion limit}
\author{Noemi David\thanks{Universite Claude Bernard Lyon 1, CNRS, Ecole Centrale de Lyon, INSA Lyon, Université Jean Monnet, 
Institut Camille Jordan, UMR5208, 43 bd du 11 novembre 1918, 69622 Villeurbanne Cedex,
France. (ndavid@math.univ-lyon1.fr; santambrogio@math.univ-lyon1.fr)}\and{Filippo Santambrogio\footnotemark[1]}\and Markus Schmidtchen\thanks{Institute of Scientific Computing, Technische Universit\"at Dresden Zellescher Weg 12-14, 01069 Dresden, Germany. (markus.schmidtchen@tu-dresden.de).}}

\maketitle
\begin{abstract}
\noindent
We present regularity results for nonlinear drift-diffusion equations of porous medium type (together with their incompressible limit). We relax the assumptions imposed on the drift term with respect to previous results and additionally study the effect of linear diffusion on our regularity result (a scenario of particular interest in the incompressible case, for it represents the motion of particles driven by a Brownian motion subject to a density constraint). Specifically, this work concerns the $L^4$-summability of the pressure gradient in porous medium flows with drifts that is stable with respect to the exponent of the nonlinearity, and $L^2$-estimates on the pressure Hessian (in particular, in the incompressible case with linear diffusion we prove that the pressure is the positive part of an $H^2$-function).  
\end{abstract}

\vskip .4cm
\begin{flushleft}
    \noindent{\makebox[1in]\hrulefill}
\end{flushleft}
	\newline\textit{Keywords and phrases.}  \\[-2.em]
\begin{flushright}
    \noindent{\makebox[1in]\hrulefill}
\end{flushright}


\section{Introduction} 
The paper is concerned with conservation laws for a density, $n=n(x,t)$, at location $x\in\R^d$ at time $t>0$,  of the form
\begin{subequations}
\begin{equation}
    \label{eq: main-active-motion}
    \partialt{n}=\nabla \cdot (n (\nabla p + \nabla V)) + \nu \Delta n,
\end{equation} 
where the first term on the right-hand side models the density's response to the pressure, $p$, and the penultimate term its response to an external drift with potential $V= V(x,t)$. The final term describes random motion with diffusivity $\nu\geq 0$. Here, the pressure is directly related to the density through the constitutive law
\begin{equation*}
    p= P(n):=n^\gamma,
\end{equation*}
where $\gamma>1$ models the stiffness of the pressure law. A quick computation shows that $p$ satisfies
\begin{align}
    \label{eq: pressure-active-motion}
    \partialt{p}=\gamma p \Delta (p + V) +\nabla p \cdot\nabla (p+V)  -\nu \frac{\gamma}{\gamma-1}\frac{|\nabla p|^2}{p} + \nu \Delta p.
\end{align}
\end{subequations}
We are interested in two distinct scenarios -- the case $\nu > 0$, corresponding to a model with active motion, and the case $\nu=0$, corresponding to a purely mechanical model, such that the density evolves according to 
\begin{subequations}
\begin{equation}
    \label{eq: main}
    \partialt{n}=\nabla \cdot (n \nabla (p + V)).
\end{equation}
In this case, the pressure satisfies the equation
\begin{align}
    \label{eq: pressure}
    \partialt{p}=\gamma p \Delta (p +V) +\nabla p \cdot\nabla(p+ V).
\end{align}
\end{subequations}
It is well known that, when $\gamma\to\infty$, these equations generate free boundary problems of Hele-Shaw type. This asymptotic is known as hard-congestion or incompressible limit and has attracted a lot of attention in the last decade, especially for its applications to crowd behaviour and 
living tissue modelling.

The aim of this paper is to separately study the cases $\nu>0$ and $\nu =0$ and establish integral estimates on the first and second-order derivatives of the pressure $p$ which are uniform with respect to the stiffness parameter $\gamma$. This will include estimates on $|\nabla p|^4$ and $|D^2p|^2$. Specifically, great attention will be dedicated to deriving such uniform bounds under relaxed assumptions on the potential, $V:\R^d\times\R\to\R$, compared to the those already present in the literature.

\subsection{Literature review and motivations from applied sciences} 

\paragraph{Motivations -- modelling of crowd behaviour.}
A crucial aspect of modelling pedestrian motion are congestion effects since it is clear that, at any point in space and time, the density of individuals cannot surpass a maximum threshold. This translates directly into imposing an upper bound on the density, usually $n\leq 1$, which is referred to as \textit{hard congestion} effect \cite{MBSV2011, MRS2010}. Clearly, solutions to the equations presented above, \eqref{eq: main-active-motion} and \eqref{eq: main}, do not guarantee the preservation of this constraint. Although the pressure acts to prevent congestion through Darcy's law, the (compressible) constitutive law $p=n^\gamma$ does not necessarily enforce $n\leq 1$ at all times. There is however a well-known link between compressible (or \textit{soft congestion}) models such as \eqref{eq: main} and their hard-congestion counterparts. In the so-called \textit{incompressible limit}, $\gamma\to \infty$, the limit pressure satisfies an inclusion, determined by a monotone multi-function $P$:
\begin{equation}
    p_\infty \in P(n_\infty), \quad\mbox{ where } P(n)=\begin{dcases}
        \{0\}, &\text{ for } n<1,\\
        [0,+\infty)  &\text{ for } n=1,\\
       \emptyset  &\text{ for } n>1.
    \end{dcases}
\end{equation}
This graph relation prevents the limit density from exceeding the value $1$ and can also be rewritten as follows
\begin{equation}\label{eq: constraint}
       p_\infty(1-n_\infty)=0, \;\quad \text{and} \quad 0\leq n_\infty\leq 1.
\end{equation}
Moreover, the limit density and pressure of solutions to Eq.~\eqref{eq: main} are expected, at least at a formal level, to satisfy
\begin{equation}
\begin{aligned}\label{eq: HS with drift}
    \partialt{n_\infty} =\Delta  p_\infty + \nabla \cdot (n_\infty \nabla V).
    \end{aligned}
\end{equation}
The interpretation of problem \eqref{eq: constraint}-\eqref{eq: HS with drift} as a macroscopic model of crowd motion was first suggested by one of the authors together with Maury and Roudneff-Chupin in \cite{MRS2010}. They consider the conservation law
\begin{equation}\label{eq: MRS}
    \partialt n + \nabla\cdot(n \boldsymbol{u})=0,
\end{equation}
where the velocity field is defined as $\boldsymbol{u}= P_{C_n} \boldsymbol{U}$, where $\boldsymbol{U}$ represents the spontaneous velocity that individuals would have in the absence of others, and $P_{C_n}$ denotes a projection operator onto the space of admissible velocities, namely the space of velocities that preserve the constraint $n\leq 1$. When $\boldsymbol{U}$ is of the form $\boldsymbol{U}=-\nabla V$, the authors show that \eqref{eq: MRS} corresponds to the 2-Wasserstein gradient flow associated with the functional
\begin{equation}\label{energy infty}
    \curlyE_\infty(n)=\begin{dcases}
        \int_{\R^d} V n\dx x &\text{ for } n\leq 1,\\
        +\infty &\text{ otherwise.}
    \end{dcases}
\end{equation}
The velocity field that preserves the density constraint is, indeed, of the form $\boldsymbol{u}=-\nabla V -\nabla p$, where $p$ is an $L_t^2H_x^1$ function that satisfies \eqref{eq: constraint}. Formally, the gradient flow associated to \eqref{energy infty} can be approximated by a porous medium-type equation as $\gamma\to \infty$, namely the gradient flow corresponding to the following energy functional
\begin{equation}
    \label{eq: energy}
    \curlyE_\gamma(n)= \int_{\R^d} n V \dx x + \frac{1}{\gamma+1}\int_{\R^d} n^{\gamma+1} \dx x.
\end{equation}
The rigorous derivation of this limit was first proven in \cite{AKY14}, where the authors prove that the solution to \eqref{eq: main} converges to the gradient flow associated to $\curlyE_\infty$ in the 2-Wasserstein distance. Moreover, they compute the first explicit rate of convergence which was later improved in the $H^{-1}$ norm \cite{DDP2021}. 

Models including sources and sinks in addition to convective effects have also been approached, \cite{DS21, KPW19}, and will be discussed later on in the paper.
Uniqueness results for problem \eqref{eq: constraint}-\eqref{eq: HS with drift} can be found in \cite{Igb2021, DM2016, DS21}.

\smallskip

\paragraph{Motivations -- tissue growth models.}
Besides crowd motion, the stiff-pressure limit of porous medium equations or systems has been widely applied to biologically motivated models, specifically, to living tissue. Indeed, Darcy's law has been successfully used to encode the cells' motion down pressure gradients. 
Unlike in the description of crowd motion, these models involve reaction terms which are usually pressure-related. 
The incompressible limit leads to Hele-Shaw-type problems where proliferation terms are responsible for the appearance of positive-pressure regions which drive the motion of the boundary. 
The first result on the incompressible limit for tumour growth models was proposed in \cite{PQV} by Perthame, Quir\'os, and Vazquez in 2014. The authors consider an equation of the form \eqref{eq: main} without accounting for convective effects, including reactions of the form $n G(p)$, where the growth rate, $G$, is a decreasing function of the pressure.
In order to study the limit $\gamma\to \infty$, the authors establish uniform bounds, specifically $BV$-bounds for both $n$ and $p$ and $\nabla p\in L^2_{x,t}$. Moreover, they prove the lower bound $\Delta p + G(p) \gtrsim -1/\gamma t$. This control is a variation of the fundamental estimate of the porous medium equation, also known as the celebrated Aronson-B\'enilan estimate \cite{AB}, and implies that $\Delta p$ is uniformly controlled as a bounded measure which is sharp since the gradient of the pressure of compactly supported solutions exhibits jumps at the free boundary.
From the seminal work \cite{PQV} research has branched out to several model extensions involving, for instance, nutrient concentrations \cite{DP21}, Brinkman's law \cite{DeSc, DEBIEC2020}, and singular pressure laws \cite{hecht2017incompressible, DHV}. 
In particular, let us briefly discuss a model including random motion on the macroscopic scale addressed in \cite{PQTV14} which is pertinent to our work. The authors study the same tissue growth model as the one in \cite{PQV} including also active motion
\begin{equation}\label{eq: pqtv}
    \partialt{n} = \nabla\cdot(n \nabla p) + \nu \Delta n + nG(p).
\end{equation}
They establish the incompressible limit as $\gamma \to \infty$ obtaining a Hele-Shaw free boundary problem in the same fashion as done in \cite{PQV}. However, the inclusion of active motion introduces crucial differences between the two models. In particular, the density $n(x,t)$ possesses better regularity properties coming from the uniform parabolicity of the equation. 
Indeed, the authors show that not only $\nabla p \in L^2$, but also $\nabla n \in L^2$. This is true uniformly in $\gamma$ -- hence also for the limits $p_\infty$ and $n_\infty$. Let us note that both the density and the pressure gradients exhibit jumps at the free boundary. However, as observed by the authors, in the sum $p_\infty + \nu n_\infty$ these singularities cancel out, and the sum actually belongs to $H^2$. It is furthermore shown that both $n_\infty$ and $p_\infty$ are continuous in space for almost every time $t>0$, which depicts a picture radically different from the case $\nu =0$. For the latter, it is known that characteristic functions of bounded moving domains are solutions of the density equation. 
It is part of the objectives of this paper to show that these properties also hold true when including a convective term into the equation, under mild assumptions on the potential $V(x,t)$.

The stiff-pressure limit was also studied for tumour growth models including local and non-local drift terms. In \cite{KPW19} the authors consider a convective-reaction-porous medium equation where the external drift  $\boldsymbol{\mathrm{b}}(x,t)$ is assumed to be $C^1_tC^3_x$ and the reaction is linear, namely $n f(x,t)$. Assuming $f- \nabla\cdot \boldsymbol{\mathrm{b}}>0$ the authors study the incompressible limit and prove convergence towards a free boundary problem using a viscosity solution approach.

The case $\boldsymbol{\mathrm{b}}=-\nabla V$ with nonlinear pressure-dependent reaction, $nG(p)$, was addressed by two of the authors in \cite{DS21} in a distributional solutions framework. Here, under substantially relaxed assumptions on the external drift, the authors establish the so-called \textit{complementarity relation} (\ie, the equation satisfied by $p_\infty$) by means of an $L^3$-Aronson-B\'enilan-type estimate and a uniform bound on $\nabla p$ in $L^4.$ 

The incompressible limit $\gamma\to\infty$ has also been analysed when the porous medium term is complemented by nonlocal drift effects, $\nabla W\star n$, rather than local, $\nabla V$, as for instance in \cite{CKY18, perthamePKS} for Newtonian potentials. We refer the reader to \cite{CT2020} for a wider class of interaction potentials, see also \cite{CG2021, CCY2019} and references therein.

\smallskip
\noindent

\paragraph{The gradient estimate in $\boldsymbol{L^4}$.} In the study of porous medium models such as \eqref{eq: main} a question that has attracted major attention is that of establishing regularity estimates on the pressure gradient $\nabla p$. The H\"older and Lipschitz regularity theory of the solution to the standard porous medium is well-established -- we refer the reader to the monograph by V\'azquez \cite{Vaz07}. However, obtaining refined estimates that are uniform in $\gamma$ has been addressed only quite recently. To the authors' best knowledge, the first such result in the context of nonlinear diffusion equations applied to tissue growth models is due to Mellet, Perthame, and Quir\'os \cite{MePeQu}. The authors address the same model studied in \cite{PQV}, namely Eq.~\eqref{eq: pqtv} with $\nu=0$. For this model, the authors prove that $\nabla p$ belongs to $L^4(0,T;L^4(\R^d))$ uniformly with respect to $\gamma$. Such a bound is obtained by testing the equation on the pressure by $-\Delta p + G(p)$ and exploiting the modified version of the Aronson-B\'enilan estimate obtained in \cite{PQV}. Later, in \cite{DP21}, the authors show the uniform $L^4$-boundedness of the pressure gradient for a system where the density equation is coupled with an equation for a nutrient concentration. Here, the nature of the coupling makes the Aronson-B\'enilan estimate in \cite{PQV} fail calling for a more general argument independent of such a bound. Moreover, in this work, the authors show that such a bound is sharp, in that there exists an example -- the so-called focusing or Graveleau solution \cite{AG}  --  for which $L^p$-norms blow up as $\gamma\to \infty$ whenever $p>4$. This property was also numerically investigated in \cite{davidruan2021}. The case involving convective effects has been considered in \cite{DS21} where a uniform $L^4$-bound is obtained on $\nabla p$ under the suitable assumption on the potential $V(x,t)$.  
Let us also mention that nowadays such an $L^4$-control is also known for nonlinear cross-diffusion systems and structured models, see \cite{Jac21, Dav23}. Similar estimates have also been established for fluid models \cite{AB20}.

\subsection{Our contribution and summary of the strategy}

\paragraph{Goals and motivations of the paper.}
As detailed in the introduction, the first uniform $L^4$-estimate on the pressure gradient for porous medium equations including convective effects was recently obtained by two of the authors in \cite{DS21}. This result was achieved under the following assumptions imposed on the potential $V(x,t)$ 
\begin{alignat}{2}
\label{assum: BV}
 \nabla \prt*{\partial_t V} &\in L_{\mathrm{loc}}^1((0,\infty); L_{\mathrm{loc}}^\infty(\R^d)), \qquad 
        \Delta \prt*{\partial_t V} &&\in L_{\mathrm{loc}}^1((0,\infty); L_{\mathrm{loc}}^1(\R^d)),\\[0.8em]
  \label{assum: D2V infty}       D^2 V &\in  L_{\mathrm{loc}}^{\infty}(0,T;L_{\mathrm{loc}}^\infty(\R^d)), 
    \qquad\qquad\quad     \grad V &&\in L_{\mathrm{loc}}^{\infty}(0,T;L_{\mathrm{loc}}^\infty(\R^d)).
\end{alignat}
The aim of this paper is to further relax these conditions and to furnish an analogous result for model including active motion, Eq.~\eqref{eq: main-active-motion}.

The relaxation of the regularity that needs to be imposed on $V$ is not only interesting in its own right and for its substantial application to the incompressible limit but it also represents a first promising step towards the understanding of more involved systems of equations. In particular, in \cite{San12} one of the authors proposed a mean-field game where a continuity equation on the density is coupled with a Hamilton-Jacobi equation.
This system is the first example of a non-variational mean-field game that takes into account the strong density constraint, $n\leq 1$, such as the one obtained in the incompressible limit of porous medium-type equations. Due to the nonvariational formulation of the problem, the existence of solutions for this system remains uncharted territory to this day. The presence of active motion, which was not considered in \cite{San12}, could mitigate the difficulties of proving the existence of solutions thanks to the additional regularity provided by the parabolic regularisation. The system then reads
\begin{equation}\label{MFG} 
 \begin{dcases}
    \partialt n -\Delta n + \nabla \cdot (n(\nabla \varphi-\nabla p))=0,
 \\[0.5em]
    \partialt \varphi + \Delta \varphi + \frac{|\nabla \varphi|^2}{2} - \nabla \varphi \cdot \nabla p=0,\\[0.5em]
    p\ge 0, \ p(1-n)=0. 
\end{dcases}
\end{equation}
The value function, $\varphi$, the macroscopic density of players, $n$, and the pressure, $p$, represent the three unknowns of the system, which is closed with the saturation condition on $n$ and $p$. The backward-forward structure of the system is the key feature of mean-field games -- the forward equation on the density describes the movement of the crowd, while the backward equation on the adjoint variable $\varphi$ represents the backward reasoning strategy of the players.
Even in the presence of diffusion, the question of proving the existence of weak solutions remains open. One of the main difficulties is the nonlinear term $\nabla\varphi\cdot\nabla p$.  It is easy to obtain $L^2$ bounds on each of the gradients separately, but this only provides weak convergence while at least one of the factors should be proven to converge strongly. The stronger bounds on the pressure derivatives established in this paper could exactly be used to prove strong compactness in $L^2$ of $\nabla p$. 

\paragraph{Summary of the strategy.}
Our analysis relies on exploiting the evolution equation satisfied by the pressure for any fixed $\gamma>0$.
In the case $\nu =0$, the pressure equation has been widely used in the literature to establish fundamental regularity properties of solutions. To obtain the uniform $L^4$-estimate on the pressure gradient we will compute the time evolution of $|\nabla p|^2$ using Eq.~\eqref{eq: pressure}, following the technique in \cite{MePeQu, DP21, DS21}. 
This will allow us to obtain a uniform estimate on the $L^2$ norm of $p D^2p$ which, together with a $L^\infty$-bound of $p$, controls the $L^4$ norms of the gradient through a simple functional inequality. Let us remark that it would be futile to try to establish an $L^2$-estimate on $D^2p$ not weighted by $p$ since the solution's gradient exhibits jumps on the free boundary $\partial\{p>0\}$.

In order to relax the assumptions on $V$ we will, however, employ a different treatment of the drift terms with respect to \cite{DS21}. In particular, our results do not rely on the uniform $BV$-bounds on the pressure established in \cite{DS21}. Therefore, we will no longer need assumption~\eqref{assum: BV} which was used to infer $\partial_t p\in L^1(0,T;L^1(\R^d)$. Moreover, we are able to improve the estimate in such a way that the $L^\infty$-control on the second derivatives of $V$, assumption~\eqref{assum: D2V infty}, can be replaced by an $L^2$-control.

In the case $\nu >0$, working with the pressure equation does not seem to be as effective as in the degenerate case. Indeed, it is unclear how to treat the additional term appearing in \eqref{eq: pressure-active-motion} which is proportional to $|\nabla \sqrt{p}|^2$. For this reason, as done in \cite{PQTV14}, rather than dealing with $p=n^\gamma$ we focus on the following quantity 
\begin{equation}
	\label{eq:defn-Sigma}
  \Sigma(x,t)= \sigma(n): = \frac{\gamma}{\gamma+1} n^{\gamma+1} + \nu n.
\end{equation} 
Computing the time evolution of $|\nabla \Sigma|^2$, we will show that $\Sigma \in L^2(0,T; H^2(\R^d))$ uniformly in $\gamma$ also in the presence of drifts. Moreover, we will show how such an estimate yields a control of $\Sigma$ in $L^4(0,T; W^{1,4}(\R^d))$, uniformly in $\gamma$.

Finally, we also show that the bounds on the Hessian of $\Sigma_\infty$ hold under even milder assumptions on $V$. This is possible since the $L^\infty$-bound on $p_\infty$ actually depends only on $\|V\|_\infty$, while the uniform (in $\gamma$) bound on the pressure requires stronger controls of the potential's derivatives.

\subsection{Outline of the rest of the paper}
The rest of the paper is organised as follows: in Section~\ref{sec: main results} we present the assumptions and main results, namely the uniform boundedness of the pressure gradient in $L^4$ for both equations~\eqref{eq: pressure} and \eqref{eq: pressure-active-motion}. Section~\ref{sec: preliminaries} is devoted to presenting some technical identities and inequalities that will be used throughout the paper. In Section~\ref{sec: pme drift}, we prove the main result for the equation without active motion, Theorem~\ref{thm: L4 no diff}, while the case $\nu>0$, Theorem~\ref{thm: L4 diff}, is treated in Section~\ref{sec: active motion}.

\section{Assumptions and main results}\label{sec: main results}
We now state the assumptions and main results of the paper for the cases $\nu=0$ and $\nu>0$, respectively.

\subsection{The advection-porous medium equation.} 
For $\nu=0$, we make the following assumptions on the initial data, $n_0(x)$, and its associated pressure $p_0(x)=(n_0(x))^\gamma$, 
\begin{equation}\label{eq: assump id}
\begin{split}   0\leq n_0, p_0 &\in L^1(\R^d)\cap L^\infty(\R^d), \quad p_0\in H^1(\R^d).\\[0.3em]
\supp(n_0)&\subset K,
\end{split}
\end{equation}
where $K \subset \R^d$ is a compact subset of $\R^d$.
We impose the following conditions on the potential $V(x,t)$:
\begin{equation}\label{eq: assum V}
\begin{aligned}
    V&\in L_{\mathrm{loc}}^\infty(0,T;L^\infty(\R^d))\cap L_{\mathrm{loc}}^2(0,\infty;H^2(\R^d)),\\[0.5em]
    \partial_t V &\in L^1_{\mathrm{loc}}(0,\infty; L^\infty(\R^d)) \cap  L_{\mathrm{loc}}^{4/3}(0,\infty, W^{1,4/3}(\R^d)).
  \end{aligned}
\end{equation} 
The $L_t^1 L_x^\infty$-control on $\partial_t V$ is needed in order to prove the uniform boundedness of the pressure through a maximum principle argument. It is immediately clear that the assumptions on $\nabla \partial_t V$ are much weaker than the ones in \eqref{assum: D2V infty}. In fact, as mentioned in the previous section, we no longer need $BV$-control in time of the pressure which required the higher regularity on $V$. However, some integrability of this term is still necessary to treat the term $\int \partial_t p \Delta V$ which appears in the time evolution of $|\nabla p|^2$.

\begin{theorem}[$L^4$-estimate on the gradient for $\nu =0$.]\label{thm: L4 no diff}
Let $(n_0, p_0)$, $V=V(x,t)$ satisfy assumptions~\eqref{eq: assump id}, \eqref{eq: assum V}. For all $T>0$, there exists a positive constant $C$, independent of $\gamma$, such that the solution, $p(x,t)$, of Eq.~\eqref{eq: pressure} with initial data $p_0$, satisfies
\begin{equation*}
      \int_0^T\!\!\int_{\R^d} p |D^2 p|^2 \dx x \dx t +      \int_0^T\!\!\int_{\R^d} |\nabla p|^4 \dx{x}\dx{t} \leq C.
\end{equation*}
\end{theorem}

\smallskip

\subsection{Including active motion -- the non-degenerate case.} 
We now discuss the assumptions and main results for the case with non-vanishing random motion. 
We make the following assumptions on the initial data
\begin{equation}\label{eq: assump id diffusion}
   0\leq n_0,  p_0\in L^1(\R^d)\cap L^\infty(\R^d)\cap H^1(\R^d).
\end{equation}
Let us remark that definition \eqref{eq:defn-Sigma} and the above assumptions imply 
\begin{equation*}
    \nabla \Sigma_0 = n_0 \nabla p_0 + \nu \nabla n_0\in L^2(\R^d).
\end{equation*}

We make the following assumptions on the potential $V(x,t)$:
\begin{equation}\label{eq: assum V diffusion}
\begin{aligned}
    V&\in L^\infty(0,\infty;L^\infty(\R^d))\cap L^2(0,\infty;H^2(\R^d)),\\[0.5em]
  \partial_t V &\in  L^2(0,\infty; L^2(\R^d)),\\[0.5em]
  \nu \Delta V + \partial_t V &\in L^1 (0,\infty; L^\infty(\R^d)).
  \end{aligned}
\end{equation}
Let us emphasise that in this case, no control on $\nabla \partial_t V$ is needed. In fact, the additional regularity coming from the uniform parabolicity permits a different way of treating $\int \partial_t \Sigma \Delta V$ only requiring $\partial_t V \in L^2_{t}L^2_x$.

We now state the main results concerning the solution of \eqref{eq: main-active-motion}.

\begin{theorem}[Uniform bounds for $\nu> 0$]\label{thm: L4 diff}
Let $(n_0, p_0)$, $V=V(x,t)$ satisfy assumptions~\eqref{eq: assump id diffusion}, \eqref{eq: assum V diffusion}. For all $T>0$, there exists a positive constant $C$, independent of $\gamma$, such that the solution, $n(x,t)$, to \eqref{eq: main-active-motion} with initial data $n_0$, 
satisfies  
\begin{equation}
         \int_0^T\!\!\int_{\R^d} |\Delta \Sigma|^2 \dx x \dx t+   \int_0^T\!\!\int_{\R^d} |\nabla \Sigma|^4 \dx x \dx t \leq C,
    \end{equation}  
where $\Sigma=\sigma(n)$ is defined in \eqref{eq:defn-Sigma}.
\end{theorem}

\bigskip

In the case of active motion, $\nu>0$, we are actually able to further weaken the assumptions on the potential needed to show the $L^4$ bound of $\nabla p_\infty$.
To do so, we first prove that, as $\gamma\to\infty$, the solution to Eq.~\eqref{eq: main-active-motion} converges to $n_\infty$, the solution of
\begin{equation}\label{eq: limit}
\begin{aligned}
    &\partialt{n_\infty} =\Delta \Sigma_\infty + \div(n_\infty \nabla V),\\[0.5em]
    &p_\infty(1-n_\infty)=0,
\end{aligned}
\end{equation}
where $\Sigma_\infty=\nu n_\infty + p_\infty$ is the limit of $\Sigma_\gamma$. The uniqueness of solutions to the limit equation was proven in \cite{DM2016}.
In order to pass to the limit we need stronger assumptions on the potential $V$ (which, for the sake of readability, will be stated in Section~\ref{sec: convergence}, see \eqref{assum V: convergence}). However, the final estimates on $\Delta\Sigma_\infty$ and $\nabla \Sigma_\infty$ will be independent of such assumptions, which we will remove through a regularisation argument, as detailed in Section~\ref{sec: convergence}.
\begin{theorem}[Convergence]\label{thm:convergence}
Let $V$ satisfy assumption~\eqref{assum V: convergence}. Then there exist $n_\infty$ and $\Sigma_\infty$ such that, up to a subsequence,
    \begin{align*}
        n_\gamma \to n_\infty &\text{ strongly in } C(0,T;L^2(\R^d)),\\[0.5em]
        \Sigma_\gamma\to \Sigma_\infty &\text{ strongly in } L^2(0,T;H^1(\R^d)),
    \end{align*}
as $\gamma \to \infty$.
\end{theorem}

\bigskip

Finally, we state the assumptions and the result on the boundedness of the $L^2_tH^2_x$-norm and the $L^4_tW^{1,4}_x$-norm of the limit $\Sigma_\infty$. We assume
\begin{equation}\label{assum V for p_infty}
\begin{aligned}
    V&\in L^\infty(0,\infty;L^\infty(\R^d))\cap L^2(0,\infty;H^2(\R^d)),\\[0.5em]
    \partial_t V &\in L^2(0,\infty; L^2(\R^d)).
  \end{aligned}
\end{equation}

 \begin{theorem}[Bounds for $\Sigma_\infty$]\label{thm: L4 diff infty}
Let $(n_0, p_0)$, $V=V(x,t)$ satisfy assumptions~\eqref{eq: assump id diffusion} and \eqref{assum V for p_infty}, and let $(n_\infty,\Sigma_\infty)$ be a solution to \eqref{eq: limit}. For all $T>0$, there exists a positive constant $C$ such that
\begin{equation*}
         \int_0^T\!\!\int_{\R^d} |\Delta \Sigma_\infty|^2 \dx x \dx t+   \int_0^T\!\!\int_{\R^d} |\nabla \Sigma_\infty|^4 \dx x \dx t \leq C.
    \end{equation*}  
\end{theorem}

\begin{corollary}\label{cor: L4 diff}
Let $(n_0, p_0)$, $V=V(x,t)$ satisfy assumptions~\eqref{eq: assump id diffusion} and \eqref{assum V for p_infty}, and let $p_\infty$ be the pressure associated to Eq.~\eqref{eq: limit}. There exists a positive constant $C$ such that
    \begin{equation*}
        \int_0^T\!\!\!\int_{\R^d} |\nabla p_\infty|^4 \dx x \dx t\leq C.
    \end{equation*}
\end{corollary}

\section{Preliminaries}
\label{sec: preliminaries}

We collect some useful functional identities and inequalities that will be used extensively in the proofs of the main results. 
\begin{proposition}
    \label{prop:identities}
    Given two functions $g, h \in H^{2}(\R^d) \cap L^{\infty}(\R^d)$, we have
    \begin{equation}
        \label{eq:identity-equality}
        \int_{\R^d} |\nabla g|^2\Delta g \dx x = -\frac 23 \int_{\R^d} g |\Delta g|^2 \dx x+ \frac 23 \int_{\R^d} g |D^2 g|^2 \dx x,
    \end{equation}
    as well as
    \begin{equation}\label{inequality L4}
        \int_{\R^d} |\nabla g|^4 \dx x \leq 8 \int_{\R^d} g^2 |\Delta g|^2 \dx x+  4\int_{\R^d} g^2 |D^2 g|^2 \dx x,
    \end{equation}
    and
    \begin{equation}
    	\label{eq:identity-3}
		\int_{\R^d}   \Delta h |\nabla g|^2 \dx x=\int_{\R^d}  h |D^2 g|^2  -   h |\Delta g|^2 +  \nabla g  D^2 h \nabla g \dx x.
    \end{equation}
\end{proposition}
\begin{proof}
Using integration by parts, we have
    \begin{align*}
        \int_{\R^d}  |\nabla g|^2 \Delta g  \dx x
        &= \int_{\R^d}  g \Delta |\nabla g|^2   \dx x \\
        &= \int_{\R^d}  2g |D^2 g|^2 + 2g \nabla g \cdot \nabla \Delta g  \dx x\\
        &= \int_{\R^d} 2g |D^2 g|^2 - 2 |\nabla g|^2 \Delta g - 2 g |\Delta g|^2  \dx x.
    \end{align*}
Upon rearranging, we obtain
    \begin{align*}
        \int_{\R^d} 3 |\nabla g|^2 \Delta g  \dx x = \int_{\R^d} 2g|D^2g|^2 - 2 g |\Delta g|^2 \dx x,
    \end{align*}
which proves the first statement, \eqref{eq:identity-equality}.
To prove \eqref{inequality L4}, we compute
\begin{align*}
    \int_{\R^d} |\nabla g|^4 \dx x &= - \int_{\R^d} g \nabla \cdot(\nabla g |\nabla g|^2) \dx x\\
    &= - \int_{\R^d} g \Delta g |\nabla g|^2 \dx x - \frac 12\int_{\R^d} \nabla g^2 \cdot \nabla |\nabla g|^2 \dx x\\
    &=- \int_{\R^d} g \Delta g |\nabla g|^2 \dx x + \frac 12 \int_{\R^d} g^2 \Delta |\nabla g|^2 \dx x\\
    &=- \int_{\R^d} g \Delta g |\nabla g|^2 \dx x + \int_{\R^d} g^2 \nabla g \cdot \nabla \Delta g \dx x + \int_{\R^d} g^2 |D^2 g|^2 \dx x\\
    &= - 3 \int_{\R^d} g \Delta g |\nabla g|^2 \dx x -  \int_{\R^d} g^2 |\Delta g|^2 +  \int_{\R^d} g^2 |D^2 g|^2\dx x.
\end{align*}
Upon using Young's inequality we obtain
\begin{align*}
     \int_{\R^d} |\nabla g|^4 \dx x & \leq  \prt*{\frac{3C}{2} - 1} \int_{\R^d} g^2 |\Delta g|^2 \dx x + \frac{3}{2C}\int_{\R^d} |\nabla g|^4 \dx x +\int_{\R^d} g^2|D^2 g|^2 \dx x.
    \end{align*}
Setting $C=2$ and rearranging, we find
    \begin{align*}
        \frac14 \int_{\R^d} |\nabla g|^4  \dx x \leq \int_{\R^d} 2 g^2 |\Delta g|^2 \dx x  + \int_{\R^d}  g^2 |D^2 g|^2 \dx x,
    \end{align*}
which yields \eqref{inequality L4}. For the final statement, we observe that
    \begin{align*}
     \int_{\R^d} \Delta h |\nabla g|^2 \dx x
        &= \int_{\R^d} 2 h (\nabla g \cdot \nabla \Delta g)  + 2 h |D^2 g|^2 \dx x\\
        &= \int_{\R^d} - 2   h |\Delta g|^2  - 2   \nabla h \cdot \nabla g \Delta g  +2 h |D^2 g|^2  \dx x\\
        &= \int_{\R^d}- 2 h |\Delta g|^2  + 2 \nabla g \cdot (D^2 h \nabla g +  D^2 g \nabla h) + 2 h |D^2 g|^2 \dx x.
    \end{align*}
Rearranging the terms in the last line, we obtain
    \begin{align*}
       \int_{\R^d} \Delta h |\nabla g|^2 \dx x=  \int_{\R^d} 2h |D^2 g|^2 - 2  h |\Delta g|^2 + 2  \nabla g \cdot D^2 h \nabla g -|\nabla g|^2 \Delta h \dx x,
    \end{align*}
which directly implies \eqref{eq:identity-3}.
\end{proof}

\section{The advection-porous medium equation}\label{sec: pme drift}
 
Now we state an important result that allows us to establish a uniform $L^\infty$-control on the pressure. This uniform bound will be obtained under different assumptions for the uniformly parabolic case $\nu >0$, and it is therefore stated as a different lemma in Section~\ref{sec: active motion}.

\begin{lemma}\label{lemma: Linfty}
    Let $T>0$ and $V=V(x,t) \in L^2(0,T;H^1(\R^d))$ such that 
    \begin{align}
        \|\partial_t V\|_{L^1(0,T;L^\infty(\R^d))} \leq C.
    \end{align}
    Let $p(x,t)$ solve Eq.~\eqref{eq: pressure}. Then, $p \in L^{\infty}(0,T; L^\infty(\R^d))$ uniformly in~$\gamma$.
\end{lemma}
 \begin{proof}
We argue by applying the comparison principle to the equation satisfied by $f:= p + V$ which reads
\begin{align*}
    \partialt{f} = \nabla p \cdot \nabla f + \gamma p \Delta f + \partialt{V}.
\end{align*}
We have
\begin{equation*}
      \| f(t) \|_{L^\infty(\R^d)}\leq \|f_0\|_{L^\infty(\R^d)} + \left\|\partial_t V\right\|_{L^1(0,T;L^\infty(\R^d))},
\end{equation*}
and therefore
\begin{align*}
    \| p(t) \|_{L^\infty(\R^d)}\leq \|p_0\|_{L^\infty(\R^d)} + \| V(\cdot,0) \|_{L^\infty(\R^d)} - \inf_{x,t} V(x,t) + \left\|\partial_t V\right\|_{L^1(0,T;L^\infty(\R^d))}.
\end{align*}
Thus, we may conclude
\begin{equation*}
  \|p\|_{L^\infty(0,T;L^\infty(\R^d))}  \leq \|p_0\|_{L^\infty(\R^d)} + 2 \|V\|_{L^\infty(0,T;L^\infty(\R^d))}+ \|\partial_t V\|_{L^1(0,T;L^\infty(\R^d))}.
\end{equation*}
 \end{proof}

Let us recall some well-known properties of solutions to Eq.~\eqref{eq: main}, and therefore of the relative pressure that satisfies \eqref{eq: pressure}.
First of all, non-negativity is preserved; $n\geq 0$ almost everywhere in $\R^d\times(0,T)$. The density $n(x,t)$ is also uniformly bounded in $L^\infty(0,T; L^1(\R^d))$; in particular $\|n(t)\|_1= \|n_0\|_1$.
As proven in Lemma~\ref{lemma: Linfty}, the pressure is uniformly bounded in $L^\infty(0,T; L^\infty(\R^d))$ which, together with $p= n p^{(\gamma-1)/\gamma}$, implies $p\in L^\infty(0,T; L^1(\R^d))$ uniformly in $\gamma$.

Exploiting this regularity, a simple computation shows that the pressure gradient is bounded uniformly $L^2(0,T; L^2(\R^d))$. Indeed, by integrating \eqref{eq: pressure} we obtain
\begin{align*}
    \ddt \int_{\R^d} p \dx x &= (1-\gamma) \int_{\R^d} |\nabla p|^2 \dx x + (1-\gamma) \int_{\R^d} \nabla p \cdot \nabla V \dx x\\[0.5em]
    &\leq \frac{1-\gamma}{2}  \int_{\R^d} |\nabla p|^2 \dx x +\frac{\gamma-1}{2}  \int_{\R^d} |\nabla V|^2 \dx x ,
\end{align*}
and thus
\begin{align*}
   \frac{\gamma-1}{2}  \int_0^T\!\!\!\int_{\R^d} |\nabla p|^2 \dx x \dx t\leq \frac{\gamma-1}{2} \int_0^T\!\!\! \int_{\R^d} |\nabla V|^2 \dx x \dx t + \int_{\R^d} p_0 \dx x,
\end{align*}
from which we deduce $\nabla p\in L^2(0,T; L^2(\R^d))$ uniformly in $\gamma$.

\subsection{Proof of Theorem \ref{thm: L4 no diff}}
 
We now proceed to proving the first main result, namely the $L^4$ estimate on the pressure gradient. 
First of all, let us notice that by using \eqref{eq:identity-equality} in Proposition~\ref{prop:identities} it is straightforward to see that the regularity assumption on the potential  $V\in L^\infty(0,T;L^\infty(\R^d))\cap L^2((0,\infty);H^2(\R^d)),$ implies $V \in L^4((0,\infty);W^{1,4}(\R^d))$.

Setting $f:= p + V$, we rewrite \eqref{eq: pressure} as
\begin{equation}
    \label{eq:pressure_rephrased}
    \partialt p = \nabla p \cdot \nabla f + \gamma p \Delta f.
\end{equation}
We multiply Eq. \eqref{eq:pressure_rephrased}  by $-(\Delta f)$ and integrate in space and time to obtain
\begin{align}\label{eq:total}
\begin{split}
    \int_0^T\ddt \int_{\R^d} &\frac{|\nabla p|^2}{2}\dx{x}\dx{t} - \int_0^T\!\!\!\int_{\R^d}  \Delta V \frac{\partial p}{\partial t} \dx{x}\dx{t} \\[0.7em]
    &=\underbrace{-  \int_0^T\!\!\!\int_{\R^d} \nabla p \cdot \nabla f  \Delta f \dx{x}\dx{t} }_{\curlyI} - \gamma  \int_0^T\!\!\!\int_{\R^d}  p |\Delta f |^2\dx{x}\dx{t}. 
\end{split}
\end{align}
Since $p=f-V$ we have
\begin{align*}
    \curlyI=& -  \int_0^T\!\!\!\int_{\R^d} \nabla p \cdot \nabla f \Delta f\dx{x}\dx{t} \\[0.7em]
    =& \underbrace{-  \int_0^T\!\!\!\int_{\R^d}  |\nabla f|^2 \Delta f\dx{x}\dx{t}}_{\curlyI_1}  + \underbrace{ \int_0^T\!\!\!\int_{\R^d}  \nabla V \cdot \nabla f \Delta f \dx{x}\dx{t}}_{\curlyI_2}.
\end{align*}
Thanks to \eqref{eq:identity-equality} we have
\begin{align*}
    \curlyI_1 &=- \int_0^T\!\!\!\int_{\R^d}  |\nabla f|^2 \Delta f\dx{x}\dx{t} \\
    &= \frac{2}{3}  \int_0^T\!\!\!\int_{\R^d} f |\Delta f|^2\dx{x}\dx{t} - \frac{2}{3}  \int_0^T\!\!\!\int_{\R^d}  f \sum_{i,j=1}^d \left|\frac{\partial^2  f}{\partial x_i \partial x_j}\right|^2\dx{x}\dx{t} \\
    	&=\frac{2}{3}  \int_0^T\!\!\!\int_{\R^d} p |\Delta f|^2\dx{x}\dx{t} - \frac{2}{3}  \int_0^T\!\!\!\int_{\R^d} p \sum_{i,j=1}^d \left|\frac{\partial^2  f}{\partial x_i \partial x_j}\right|^2\dx{x}\dx{t}\\
    &\qquad	+\frac{2}{3}  \int_0^T\!\!\!\int_{\R^d} V |\Delta f|^2\dx{x}\dx{t} - \frac{2}{3}  \int_0^T\!\!\!\int_{\R^d} V \sum_{i,j=1}^d \left|\frac{\partial^2  f}{\partial x_i \partial x_j}\right|^2\dx{x}\dx{t}.
\end{align*} 
Using identity Eq.\eqref{eq:identity-3} of Proposition \ref{prop:identities}, the last two terms of the previous line become
\begin{align*}
  \frac{2}{3}  \int_0^T\!\!\!\int_{\R^d} &V |\Delta f|^2 -  V\sum_{i,j=1}^d \left|\frac{\partial^2  f}{\partial x_i \partial x_j}\right|^2\dx{x}\dx{t} \\[0.5em]
  &= \frac 2 3  \int_0^T\!\!\!\int_{\R^d}  ( \nabla f D^2 V  \nabla f - \Delta V |\nabla f|^2) \dx x \dx t\\[0.5em]
  &\leq \varepsilon \|\nabla f\|_4^4 +\frac C \varepsilon \|D^2V\|_2^2 + \frac C \varepsilon \|\Delta V\|_2^2\\[0.5em]
  &\leq \varepsilon \|\nabla f\|_4^4 + C,
\end{align*}
where we used Young's inequality and $D^2 V \in L^2(0,T;L^2(\R^d))$. Therefore, we get
\begin{align}
	\mathcal I_1 \leq \frac23  \int_0^T\!\!\!\int_{\R^d}  p |\Delta f|^2 - p \sum_{i,j=1}^d \abs*{\frac{\partial^2 f}{\partial x_i \partial x_j}}^2 \dx x \dx t + \varepsilon \|\nabla f\|_4^4 + C.
\end{align}

Now we proceed to estimate the term $\curlyI_2$
\begin{align*}
    \curlyI_2 &=  \int_0^T\!\!\!\int_{\R^d}  \nabla V \cdot \nabla f \Delta f \dx{x}\dx{t}\\[0.4em]
    &= -  \int_0^T\!\!\!\int_{\R^d}  \nabla f \cdot D^2V \nabla f \dx x \dx t -  \int_0^T\!\!\!\int_{\R^d}  \nabla V \cdot  D^2f \nabla f \dx{x}\dx{t}\\[0.4em]
    &\leq \frac\varepsilon 2  \int_0^T\!\!\!\int_{\R^d}  |\nabla f|^4 \dx{x}\dx{t}+ \frac{1}{2\varepsilon}  \int_0^T\!\!\!\int_{\R^d}  |D^2V|^2\dx{x}\dx{t}  -  \int_0^T\!\!\!\int_{\R^d}  \nabla V \cdot  D^2f \nabla f \dx{x}\dx{t}\\[0.4em]
    &\leq \frac \varepsilon 2 \|\nabla f\|^4_{4} + C \|D^2V\|^2_{2}  -\frac 1 2  \int_0^T\!\!\!\int_{\R^d}  \nabla V\cdot \nabla|\nabla f|^2 \dx{x}\dx{t}\\[0.4em]
     &= \frac\varepsilon 2 \|\nabla f\|^4_{4} + C  \|D^2V\|^2_{2}  +\frac 1 2  \int_0^T\!\!\!\int_{\R^d}  \Delta V |\nabla f|^2 \dx{x}\dx{t}\\[0.4em]
   &\leq \varepsilon \|\nabla f\|^4_{4} + C,
\end{align*}
where we again used the weighted Young's inequality and $V\in L^2(0,T;H^2(\R^d))$. Therefore, we obtain
\begin{align*}
    \curlyI = \curlyI_1 + \curlyI_2 \leq& \frac{2}{3}  \int_0^T\!\!\!\int_{\R^d} p |\Delta f|^2\dx{x}\dx{t} - \frac{2}{3}  \int_0^T\!\!\!\int_{\R^d} p \sum_{i,j=1}^d \left|\frac{\partial^2  f}{\partial x_i \partial x_j}\right|^2\dx{x}\dx{t} + 2\varepsilon \|\nabla f\|^4_{4} + C.
\end{align*}


Gathering all the bounds we can write Eq. \eqref{eq:total} as
\begin{align*} 
\begin{split}
  \frac12  \|\nabla p(T)\|_2^2 &+  \frac 2 3   \int_0^T\!\!\!\int_{\R^d}  p \sum_{i,j=1}^d\left|\frac{\partial^2 f}{\partial x_i \partial x_j}\right|^2\dx{x}\dx{t} +\prt*{\gamma-\frac 2 3}  \int_0^T\!\!\!\int_{\R^d}   p |\Delta f|^2\dx{x}\dx{t}\\[0.6em]
  &\leq  C+ 2\varepsilon \|\nabla f\|^4_{4} +  \int_0^T\!\!\!\int_{\R^d}  \partial_t{\nabla V}\cdot \nabla p \dx x \dx t + \frac12  \|\nabla p_0\|_2^2\\[0.6em]
  &\leq C+ C \varepsilon \|\nabla p\|^4_{4} + \frac12 \|\nabla p_0\|_2^2,
  \end{split}
\end{align*}
where we used Young's inequality with exponents $4$ and $4/3$, $\nabla V \in L^4$ and the $L^{4/3}$-bound of $\partial_t \nabla V$ assumed in \eqref{eq: assum V}. Thus, we have proven the following bound
\begin{equation}
	\label{eq:estimate-pD2f_squared}
 \begin{aligned}
\frac 12& \|\nabla p(T)\|_{L^2}^2 +	\frac 2 3  \int_0^T\!\!\!\int_{\R^d}  p \sum_{i,j=1}^d\left|\frac{\partial^2 f}{\partial x_i \partial x_j}\right|^2\dx{x}\dx{t} +\left(\gamma-\frac 2 3\right)  \int_0^T\!\!\!\int_{\R^d}  p |\Delta f|^2 \dx{x}\dx{t}\\[0.8em]
&\qquad\qquad\qquad\qquad\leq C+ C\varepsilon \|\nabla p\|^4_{4},
\end{aligned}
\end{equation}
where we used the assumption $p_0\in H^1(\R^d)$.
Using \eqref{eq:estimate-pD2f_squared} in conjunction with the uniform bounds on $D^2 V$ in $L^2(0,T;L^2(\R^d))$ and $p \in L^\infty(0,T;L^\infty(\R^d))$, we have
\begin{align}
\label{eq: pd2p}
\begin{split}
     \int_0^T\!\!\!\int_{\R^d}  p \sum_{i,j=1}^d \left|\frac{\partial^2 p}{\partial x_i \partial x_j}\right|^2 \dx{x}\dx{t}&\leq 2 \int_0^T\!\!\!\int_{\R^d}  p \sum_{i,j=1}^d\left|\frac{\partial^2 f}{\partial x_i \partial x_j}\right|^2\dx{x}\dx{t}\\
     &\qquad+2 \int_0^T\!\!\!\int_{\R^d}  p \sum_{i,j=1}^d \left|\frac{\partial^2V}{\partial x_i \partial x_j}\right|^2\dx{x}\dx{t}\\[0.5em]
    &\leq C + \varepsilon \|\nabla p\|^4_{4},
\end{split}
\end{align}
as well as
\begin{align}
\label{eq: pLapp}
\begin{split}
     \int_0^T\!\!\!\int_{\R^d}  p \left|\Delta p \right|^2 \dx{x}\dx{t}&\leq 2 \int_0^T\!\!\!\int_{\R^d}  p \left|\Delta f\right|^2\dx{x}\dx{t} +2 \int_0^T\!\!\!\int_{\R^d}  p \left|\Delta V\right|^2\dx{x}\dx{t}\\[0.7em]
    &\leq C + C\varepsilon \|\nabla p\|^4_{4}.
\end{split}
\end{align}

Using the uniform $L^\infty$-bound on $p$ in \eqref{inequality L4} in Proposition \ref{prop:identities}, we obtain
\begin{equation*}
	 \int_0^T\!\!\!\int_{\R^d}  |\nabla p|^4\dx{x} \dx t \leq C \int_0^T\!\!\!\int_{\R^d}  p |\Delta p|^2 \dx{x}\dx t +  C \int_0^T\!\!\!\int_{\R^d}  p \sum_{i,j=1}^d \left|\frac{\partial^2  p}{\partial x_i \partial x_j}\right|^2\dx{x} \dx t,
\end{equation*}
where the right-hand side is controlled by (\ref{eq: pd2p}, \ref{eq: pLapp}). We conclude 
\begin{equation*}
	\|\nabla p\|_{4}^4 \leq C\varepsilon \|\nabla p\|^4_{4} + C,
\end{equation*}
which completes the proof of the uniform $L^4$-bound of the gradient if $\varepsilon>0$ is sufficiently small. Finally, \eqref{eq:estimate-pD2f_squared} yields $p\in L^\infty(0,T; H^1(\R^d))$, and this concludes the proof.

\section{Including active motion}\label{sec: active motion}

Let us now consider \eqref{eq: main-active-motion} and \eqref{eq: pressure-active-motion}.
Defining $\Sigma$ as in \eqref{eq:defn-Sigma},  \eqref{eq: main-active-motion} becomes
\begin{align}\label{eq: main active motion NEW}
    \partialt n = \Delta \Sigma + \nabla\cdot(n\nabla V).
\end{align}
The proof of the main result heavily depends on integration by parts. Since, in this case, solutions are no longer compactly supported as in the degenerate case,  we prove the result on the torus $\T^d$ for the ease of exposition.

Before proving the main result, Theorem~\ref{thm: L4 diff}, let us state the following lemma which gives the uniform control of the $L^\infty$-norm of the pressure. 

\begin{lemma}[$\norm{p}_{L^\infty_{x,t}} \leq C_*$]\label{lemma Cs}
Let $p$ satisfy \eqref{eq: pressure} with a potential $ V$ such that $\norm{ V}_{L^\infty(0,T; L^\infty(\T^d))},$ $\norm{\nu \Delta V+\partial_t V}_{L^1(0,T; L^\infty(\T^d))}$ are bounded. There exists a constant, $C_*>0$ such that
\begin{align}
    \norm{p}_{L^\infty_{x,t}} \leq C_*, \quad \norm{n}_{L^\infty_{x,t}} \leq C_*^{1/\gamma},
\end{align} 
where $n=p^{1/\gamma}$.
\end{lemma}
\begin{proof}
Consider the equation for the pressure
\begin{align*}
        \partial_t p = \nabla p \cdot \nabla (p + V) + \gamma p (\Delta p + \Delta V) - \nu \frac{\gamma - 1}{\gamma} \frac{|\nabla p|^2}{p} + \nu \Delta p.
\end{align*}
Setting $f:= p + V$, we get
\begin{align*}
        \partial_t f &= \nabla p \cdot \nabla f + \gamma p \Delta f - \nu \frac{\gamma - 1}{\gamma} \frac{|\nabla p|^2}{p} + \nu \Delta f - \nu \Delta V  -\partial_t V\\[0.5em]
        &\leq  \nabla p \cdot \nabla f + \gamma p \Delta f + \nu \Delta f - \nu \Delta  V  -\partial_t  V.
\end{align*}
Arguing as in the proof of Lemma~\ref{lemma: Linfty}, by the comparison principle, we have 
\begin{align*}
    \|f(t)\|_{L^\infty(\T^d)} \leq \| f(0)\|_{L^\infty(\T^d)} + \norm{\nu \Delta V+\partial_t  V}_{L^1(0,T; L^\infty(\T^d))}.
\end{align*}
Then, we conclude
\begin{align*}
    \| p(t)\|_{L^\infty(\T^d)} 
    &\leq \|f(0)\|_{L^\infty(\T^d)} + \norm{\nu \Delta  V+\partial_t  V}_{L^1(0,T; L^\infty(\T^d))} - \min_{x,t}  V\\[0.5em]
    &\leq C(\norm{p_0}_{L^\infty}, \norm{ V}_{L^\infty(0,T; L^\infty(\T^d))}, \norm{\nu \Delta  V+\partial_t  V}_{L^1(0,T; L^\infty(\T^d))}) =: C_*.
\end{align*}
Since $n = p^{1/\gamma}$, we get the second bound.
\end{proof}

\subsection{Proof of Theorem~\ref{thm: L4 diff} -- uniform results in $\gamma$}

It is easy to see that $\Sigma$ satisfies the following equation
 \begin{align}\label{eq: new}
    \partial_t \Sigma = (\nu + \gamma p) (\Delta \Sigma + n \Delta  V) + \nabla \Sigma \cdot \nabla  V.
\end{align}
We multiply \eqref{eq: new} by $-(\Delta \Sigma + n  V)$ and get
\begin{align*}
    \frac12
    &\int_{\T^d} |\nabla \Sigma|^2(T) \dx x + (\nu + \gamma p) \int_0^T\!\!\int_{\T^d} |\Delta \Sigma + n \Delta  V|^2 \dx x \dx t \\[0.5em]
    &= \int_0^T\!\!\int_{\T^d} n \partial_t \Sigma \Delta  V \dx x \dx t - \int_0^T\!\!\int_{\T^d} \nabla \Sigma \cdot \nabla  V (\Delta \Sigma + n \Delta  V) \dx x \dx t+ \frac12 \int_{\T^d} |\nabla \Sigma_0|^2\dx x.
\end{align*}
Setting $S'(s) := s \Sigma'(s)$, we get
\begin{equation}
    \label{eq:DissipGradSigma}
    \begin{split}
    \frac12
    \int_{\T^d} |\nabla \Sigma|^2(T) \dx x + \frac\nu2 \int_0^T\!\!\int_{\T^d} |\Delta \Sigma + n \Delta  V|^2 \dx x \dx t &\leq  \frac{1}{2\nu} \int_0^T\!\!\int_{\T^d} |\nabla \Sigma \cdot \nabla  V|^2 \dx x \dx t\\[0.5em]
    &\qquad +\mathcal J + C,
    \end{split}
\end{equation}
where $C>0$ depends on $\|\nabla \Sigma_0\|_{L^2}$ and
\begin{align*}
    \mathcal J := \int_0^T\!\!\int_{\T^d} \partial_t S \Delta  V \dx x \dx t. 
\end{align*}
\paragraph{Treatment of $\mathcal J$.} By the definition of $S$, we have
\begin{align*}
    \mathcal J &= - \int_0^T\!\!\int_{\T^d} \Delta S \partial_t  V \dx x \dx t\\[0.5em]
    &= - \int_0^T\!\!\int_{\T^d} n \Delta \Sigma \partial_t  V \dx x \dx t - \int_0^T\!\!\int_{\T^d} \frac{|\nabla \Sigma|^2}{\Sigma'} \partial_t  V \dx x \dx t,
\end{align*}
having used $\Delta S = n \Delta \Sigma + |\nabla \Sigma|^2 / \Sigma'$. Then,
\begin{align*}
    \mathcal J&=- \int_0^T\!\!\int_{\T^d} n (\Delta \Sigma + n \Delta  V) \partial_t  V \dx x \dx t + \int_0^T\!\!\int_{\T^d} n^2 \Delta  V  \partial_t  V \dx x \dx t - \int_0^T\!\!\int_{\T^d} \frac{|\nabla \Sigma|}{\Sigma'} \partial_t  V \dx x \dx t\\[0.5em]
    &\leq \frac\alpha2 \int_0^T\!\!\int_{\T^d} |\Delta \Sigma + n \Delta  V|^2 \dx x \dx t + \frac{1}{2\alpha} \int_0^T\!\!\int_{\T^d} n^2 |\partial_t  V|^2 \dx x \dx t + \norm{n}_{L^\infty}^2 \norm{\Delta  V}_{L^2} \norm{\partial_t  V}_{L^2}\\[0.3em]
    &\qquad + \frac1\nu \int_0^T\!\!\int_{\T^d} |\nabla \Sigma|^2 |\partial_t  V| \dx x \dx t\\[0.5em]
    &\leq \frac\nu4 \int_0^T\!\!\int_{\T^d} |\Delta \Sigma +  n \Delta  V|^2 \dx x \dx t + \frac1\nu \norm{n}_{L^\infty}^2 \norm{\partial_t V}_{L^2}^2 + \norm{n}_{L^\infty}^2 \norm{\Delta  V}_{L^2} \norm{\partial_t  V}_{L^2}\\[0.3em]
    & \qquad + \frac1\nu \int_0^T\!\!\int_{\T^d} |\nabla \Sigma|^2 |\partial_t  V| \dx x \dx t.
\end{align*}
Thus, revisiting \eqref{eq:DissipGradSigma}, we get
\begin{align}
    \label{eq:dissipgradsigmarevisited2}
    \begin{split}
    \frac12 
    \int_{\T^d} |\nabla \Sigma|^2(T)& \dx x + \frac\nu4 \int_0^T\!\!\int_{\T^d} |\Delta \Sigma + n \Delta V|^2 \dx x \dx t \\[0.5em]
    &\leq
    \frac1\nu \norm{n}_{L^\infty}^2 \norm{\partial_t  V}_{L^2}^2 + \norm{n}_{L^\infty}^2 \norm{\Delta  V}_{L^2} \norm{\partial_t  V}_{L^2}+C\\[0.3em]
    &\qquad+ \frac1\nu \int_0^T\!\!\int_{\T^d} |\nabla \Sigma|^2 |\partial_t  V| \dx x \dx t +  \frac{1}{2\nu} \int_0^T\!\!\int_{\T^d} |\nabla \Sigma \cdot \nabla  V|^2 \dx x \dx t\\[0.5em]
    &\leq  
     C_*^{\frac 2\gamma} \left(\frac1\nu\norm{\partial_t  V}_{L^2}^2 + \norm{\Delta  V}_{L^2} \norm{\partial_t  V}_{L^2}\right) + C\\[0.3em] 
     &\qquad + \frac1\nu \int_0^T\!\!\int_{\T^d} |\nabla \Sigma|^2 |\partial_t  V| \dx x \dx t +  \frac{1}{2\nu} \int_0^T\!\!\int_{\T^d} |\nabla \Sigma \cdot \nabla  V|^2 \dx x \dx t.
    \end{split}
\end{align}
Let us recall that
\begin{equation*}
    \int_{\T^d} |\nabla \Sigma|^4 \dx x \leq C \|\Sigma\|_\infty^2 \int_{\T^d} |\Delta \Sigma|^2 \dx x. 
\end{equation*}
Since $n$ is uniformly bounded in $L^\infty$, using Young's inequality on the last two integrals on the right-hand side of \eqref{eq:dissipgradsigmarevisited2}, we find
\begin{align*}
    \begin{split}
    \frac12 
    \int_{\T^d} |\nabla \Sigma|^2(T) \dx x + \frac\nu4 \int_0^T\!\!\int_{\T^d} |\Delta \Sigma|^2 \dx x \dx t &\leq C +\varepsilon \int_0^T\!\!\int_{\T^d} |\nabla \Sigma|^4 \dx x \dx t \\[0.5em]
    &\leq C +\varepsilon C \int_0^T\!\!\int_{\T^d} |\Delta \Sigma|^2  \dx x \dx t ,
    \end{split}
\end{align*}
where $C=C(C_*, \norm{\partial_t  V}_{L^2}, \norm{\Delta  V}_{L^2}, \|\nabla V\|_{L^4}, \varepsilon)$. Choosing $\varepsilon$ small enough yields the bound.

\subsection{Proof of Theorem~\ref{thm:convergence} -- convergence to the limit}
\label{sec: convergence}

To prove Theorem~\ref{thm: L4 diff infty}, namely the estimates for the limit $\Sigma_\infty$ under relaxed assumptions on $V$, we will proceed from \eqref{eq:dissipgradsigmarevisited2}. In particular, instead of bounding the last two integrals, we aim at passing to the limit as $\gamma\to \infty$. In this section, we prove that the sequence $\nabla \Sigma_\gamma$ is strongly compact. This holds under stronger assumptions on the potential, see \eqref{assum V: convergence}. However, since the final estimate does not depend on such assumptions, we will later remove them in Section~\ref{sec: active motion infinity}.

\paragraph{Sequence of regularised potentials.}
Let $V$ satisfy \eqref{assum V for p_infty}. We take a sequence of smooth potentials $V_\delta:\T^d\times[0,T]\to \T^d$ such that, 
as $\delta\to 0$, we have
\begin{align*}
    \Delta  V_\delta \to \Delta V, \quad &\text{strongly in } L^2(0,T;L^2(\T^d)),\\[0.5em] 
    \partial_t V_\delta \to \partial_t V, \quad  &\text{strongly in } L^2(0,T;L^2(\T^d)),\\[0.5em] 
    \nabla V_\delta \to \nabla V, \quad  &\text{strongly in } L^4(0,T;L^4(\T^d)).
\end{align*} 
In particular, we have
\begin{equation}\label{assum V: convergence}
    \partial_t V_\delta \in L^\infty(0,T;L^\infty(\T^d)), \; \nabla V_\delta \in  L^\infty(0,T;L^\infty(\T^d)).
\end{equation}

\paragraph{A priori estimates and strong compactness.}
We will now consider the solution of Eq.~\eqref{eq: main active motion NEW} with potential $V_\delta$ (which for the sake of readability will simply be denoted by $\tilde V$). Our goal is to show the strong convergence of $\Sigma_\gamma$ to $\Sigma_\infty$ in $L^2(0,T;H^1(\T^d))$. To this end, we introduce the index $\gamma$ to indicate the sequence (or the not relabelled subsequences) of solutions. The result follows from the uniform \textit{a priori } estimates proved in the following lemmata.

\begin{lemma}[$\nabla \Sigma_\gamma$ is uniformly bounded in $L^2$]\label{lemma nabla sigma}
There exists $C>0$ independent of $\gamma$ such that
\begin{equation}
    \|\nabla \Sigma_\gamma\|_{L^2(0,T; L^2(\T^d)}\leq C.
\end{equation}
\end{lemma}

\begin{proof}
We recall the equation for $n_\gamma$
\begin{align*}
        \partial_t n_\gamma = \Delta \Sigma_\gamma + \nabla \cdot (n_\gamma \nabla \tilde V).
\end{align*}
Multiplying by $\Sigma_\gamma$ and integrating in space, we get
\begin{align*}
        \int_{\T^d} \Sigma_\gamma \partial_t n_\gamma \dx x =  -\int_{\T^d}  |\nabla \Sigma_\gamma|^2 \dx x - \int_{\T^d} n_\gamma \nabla \tilde V \cdot \nabla \Sigma_\gamma \dx x.
\end{align*}
The left-hand side is an exact derivative
\begin{align*}
        \prt*{\nu n_\gamma + \frac{\gamma}{\gamma + 1}n^{\gamma +1 }} \partial_t n_\gamma = \partialt{} \left(\frac\nu2 n_\gamma^2 + \frac{\gamma}{(\gamma + 1)(\gamma + 2)}n_\gamma^{\gamma + 2}\right).
\end{align*}
Using Young's inequality on the right-hand side, we get
\begin{align*}
        \ddt \int_{\T^d} \frac\nu2 n_\gamma^2 + \frac{\gamma}{(\gamma + 1)(\gamma + 2)}n_\gamma^{\gamma + 2} \dx x  + \frac12 \int_{\T^d} |\nabla \Sigma_\gamma|^2 \dx x \leq \frac12 \norm{\nabla \tilde V}_{L^2}^2 \norm{ n_\gamma}_{L^\infty}^2.
\end{align*}   
Integrating in time, we get
\begin{align*}
        \frac\nu2 \int_{\T^d} n_\gamma^2(t) \dx x + \frac12\int_0^T\!\!\int_{\T^d} |\nabla \Sigma_\gamma|^2 \dx x \dx t &\leq \norm{\nabla \tilde V}_{L^2}^2 T C_*^{2/\gamma} + \frac\nu2 \int_{\T^d} n_\gamma^2(0)\dx x\\[0.3em]
        &\qquad + \frac{\gamma}{(\gamma + 1)(\gamma + 2)} \int_{\T^d} n_\gamma^{\gamma + 2}(0) \dx x,
\end{align*}
which concludes the proof.
\end{proof} 

\begin{lemma}[Boundedness of $\Delta \Sigma_\gamma$ in $L^2$]
\label{lemma delta sigma}
There exists $C>0$ independent of $\gamma$, such that
\begin{equation*}
    \|\Delta \Sigma_\gamma\|_{L^2(0,T;L^2(\T^d))} \leq C.
\end{equation*} 
Thus, we have
\begin{equation*}
    \Delta \Sigma_\gamma\rightharpoonup \Delta \Sigma_\infty, \quad\text{weakly in } L^2(0,T;L^2(\T^d)).
\end{equation*}
\end{lemma}

\begin{proof}
From \eqref{eq:dissipgradsigmarevisited2} we have
\begin{align*}
    \frac 12 \int_{\T^d} |\nabla \Sigma_\gamma(T)|^2 \dx x &+ \frac \nu 4 \int_0^T\int_{\T^d}|\Delta \Sigma_\gamma +n_\gamma \Delta \tilde V|^2 \dx x \dx t\\[0.3em]
    &\leq  \frac 1 \nu C_*^{\frac 2\gamma} \|\partial_t \tilde V\|^2_{L^2}  + C_*^{\frac 2\gamma}\|\Delta \tilde V\|_{L^2}\|\partial_t\tilde V\|_{L^2}+ \frac 1 \nu \|\nabla\Sigma_\gamma\|_{L^2}^2\|\partial_t \tilde V\|_{L^\infty}\\[0.3em]
    &\qquad\qquad+ \frac{1}{2\nu} \|\nabla \tilde V\|^2_{L^\infty} \|\nabla \Sigma_\gamma\|^2_{L^2}\\[0.5em]
    &\leq C,
\end{align*}
where $C>0$ is independent of $\gamma$ by Lemma~\ref{lemma nabla sigma}. 
Thus, we have (uniformly in $\gamma$)
\begin{equation*}
       \|\Delta \Sigma_\gamma + n_\gamma \Delta \tilde  V\|^2_{L^2}\leq C.
\end{equation*}
Using the uniform boundedness of $n_\gamma$, and $\Delta \tilde V\in L^2(0,T;L^2(\T^d))$, we conclude.
\end{proof}

\begin{lemma}[$\nabla p_\gamma$ is uniformly bounded in $L^2$]
\label{lemma nabla p}
There exists $C>0$ independent of $\gamma$, such that
\begin{equation}
    \|\nabla p_\gamma\|_{L^2(0,T;L^2(\T^d))} \leq C.
\end{equation}
\end{lemma}
\begin{proof}
   Let us integrate the equation for $p_\gamma$ \eqref{eq: pressure-active-motion} to obtain
   \begin{align*}
       \ddt \int_{\T^d} p_\gamma \dx x &\leq (1-\gamma) \int_{\T^d} \nabla p_\gamma \cdot (\nabla p_\gamma +\nabla \tilde V) \dx x,
   \end{align*}
and thus, by Young's inequality
 \begin{align*}
     \int_{\T^d} p_\gamma(T) \dx x +\frac{\gamma-1}{2} \int_0^T\!\!\int_{\T^d} |\nabla p_\gamma|^2\dx x \dx t \leq  \int_{\T^d} p_\gamma(0) \dx x
 +\frac{\gamma-1}{2} \int_0^T\!\!\int_{\T^d} |\nabla \tilde V|^2 \dx x \dx t,    \end{align*}
 and the proof is complete.
\end{proof}

\begin{lemma}[Strong compactness of $n_\gamma$.]
\label{strong n}
For $1\leq q<\infty$, as $\gamma\to \infty$ we have
$$
    n_\gamma \to n_\infty, \textit{ strongly in } L^q(0,T;L^q(\T^d)),
$$
up to a subsequence. Moreover,
$n_\gamma \to n_\infty$ strongly in $C(0,T;L^2(\T^d))$.
\end{lemma}
\begin{proof}
From the proofs of Lemmata~\ref{lemma nabla sigma}-\ref{lemma delta sigma}-\ref{lemma nabla p}, we have  
\begin{equation*}
        \nabla{n_\gamma} = \frac 1 \nu \left( \nabla \Sigma_\gamma - n_\gamma\nabla p_\gamma\right) \in L^\infty(0,T; L^2(\T^d)),
    \end{equation*}
    and
    \begin{equation*}
        \partialt{n_\gamma} = \Delta \Sigma_\gamma + \nabla\cdot(n_\gamma \nabla \tilde V) \in L^2(0,T;L^2(\T^d)),
    \end{equation*}
uniformly in $\gamma$. Therefore, by Aubin-Lions' lemma, we conclude.
\end{proof}

\bigskip
    
The strong compactness of $n_\gamma$ given by Lemma~\ref{strong n} and the weak compactness of $\nabla \Sigma_\gamma$ implied by Lemma~\ref{lemma nabla sigma} are enough to pass to the limit in \eqref{eq: new} and show that $n_\infty$ and $\Sigma_\infty$ satisfy
  \begin{equation*}
        \partialt{n_\infty} =\Delta \Sigma_\infty + \nabla\cdot(n_\infty \nabla\tilde V)
    \end{equation*}
Exploiting the energy structure of the limit equation we show that the weak convergence of $\nabla \Sigma_\gamma$ can be uplifted to strong convergence.

\begin{lemma}[Strong convergence of $\nabla \Sigma_\gamma$]
As $\gamma\to\infty$, we have
\begin{equation*}
    \nabla \Sigma_\gamma \to \nabla \Sigma_\infty, \text{ strongly in } L^2(0,T; L^2(\T^d)).
\end{equation*}
\end{lemma}
\begin{proof}
Testing the equations 
\begin{equation*}
        \partialt{n_\gamma} =\Delta \Sigma_\gamma + \nabla\cdot(n_\gamma \nabla\tilde V),
    \end{equation*}
     \begin{equation*}
        \partialt{n_\infty} =\Delta \Sigma_\infty + \nabla\cdot(n_\infty \nabla\tilde V),
    \end{equation*}
by $\Sigma_\gamma$ and $\Sigma_\infty$, respectively, we have
    \begin{align*}
          \int_0^T \!\!\int_{\T^d}\Sigma_\gamma \partialt{n_\gamma} \dx x= -  \int_0^T \!\!\int_{\T^d} |\nabla \Sigma_\gamma|^2 \dx x -  \int_0^T \!\! \int_{\T^d}\nabla \Sigma_\gamma \cdot \nabla \tilde V n_\gamma \dx x,
    \end{align*}
        \begin{align*}
         \int_0^T \!\! \int_{\T^d}\Sigma_\infty \partialt{n_\infty} \dx x= -  \int_0^T \!\!\int_{\T^d} |\nabla \Sigma_\infty|^2 \dx x -   \int_0^T \!\!\int_{\T^d}\nabla \Sigma_\infty \cdot \nabla \tilde V n_\infty \dx x.
    \end{align*}
 Our goal is to show
 \begin{equation*}
     \lim_{\gamma\to \infty}\|\nabla\Sigma_\gamma\|_{L^2}= \|\nabla\Sigma_\infty\|_{L^2}.
 \end{equation*}
 From the two energy equalities, we have
     \begin{align*}
    \int_0^T \!\!\int_{\T^d} |\nabla \Sigma_\gamma|^2 \dx x \dx t=& \int_0^T \!\!\int_{\T^d}\Sigma_\infty \partialt{n_\infty} \dx x \dx t -  \int_0^T \!\! \int_{\T^d}\Sigma_\gamma \partialt{n_\gamma} \dx x \dx t +\int_0^T \!\! \int_{\T^d} |\nabla \Sigma_\infty|^2 \dx x \dx t\\[0.3em]
     &\qquad+ \int_0^T \!\!\int_{\T^d}\nabla \Sigma_\infty \cdot \nabla \tilde V n_\infty \dx x \dx t - \int_0^T \!\! \int_{\T^d}\nabla \Sigma_\gamma \cdot \nabla \tilde V n_\gamma \dx x \dx t\\[0.5em]
     \to& \int_0^T \!\!\int_{\T^d} |\nabla \Sigma_\infty|^2 \dx x \dx t,
    \end{align*}
because the two differences on the right-hand side converge to zero as $\gamma\to\infty$ by the strong compactness of $\Sigma_\gamma$ (and $n_\gamma$) and the weak compactness of $\partial_t n_\gamma$ (and $\nabla\Sigma_\gamma$), 
\end{proof}

The above argument concludes the proof of Theorem~\ref{thm:convergence}.

\subsection{Proof of Theorem~\ref{thm: L4 diff infty} -- results for the limit pressure}
\label{sec: active motion infinity}
We now prove that $\Delta \Sigma_\infty$ is bounded in $L^2(0,T; L^2(\R^d))$ and $\Sigma_\infty \in L^\infty(0,T;L^\infty(\R^d))$ -- which implies $\Sigma_\infty \in L^4(0,T;W^{1,4}(\R^d))$ -- under relaxed assumptions.
\begin{proof}[Proof of Theorem~\ref{thm: L4 diff infty}]
Finally, we may come back to \eqref{eq:dissipgradsigmarevisited2}, where we consider the equation with the regularised potential, $\tilde V$. Using the fact that $|\Delta \Sigma_\gamma + n_\gamma \Delta V|^2 = |\Delta \Sigma_\gamma|^2 + 2 \Delta\Sigma_\gamma n_\gamma \Delta \tilde V + n_\gamma^2 |\Delta \tilde V|^2$, we can revisit the inequality as follows
\begin{align*}
    \begin{split}
    \frac12 
    &\int |\nabla \Sigma_\gamma|^2(T) \dx x + \frac\nu8 \int_0^T\!\!\int_{\T^d} |\Delta \Sigma_\gamma|^2 \dx x \dx t \\[0.5em]
    &\leq  \frac\nu4 C_*^{\frac 2\gamma} \norm{\Delta \tilde V}_{L^2}^2+
    \frac1\nu C_*^{\frac 2\gamma} \norm{\partial_t \tilde V}_{L^2_{x,t}}^2 + C_*^{\frac 2 \gamma} \norm{\Delta \tilde V}_{L^2_{x,t}} \norm{\partial_t \tilde V}_{L^2_{x,t}} + C\\[0.3em]
    &\qquad+ \frac1\nu \int_0^T\!\!\int_{\T^d} |\nabla \Sigma_\gamma|^2 |\partial_t \tilde V| \dx x \dx t + \frac 1\nu \int_0^T\!\!\int_{\T^d} |\nabla \Sigma_\gamma \cdot \nabla \tilde V|^2 \dx x \dx t .
    \end{split}
\end{align*}
Letting $\gamma \to \infty$, and using lower semi-continuity of convex functionals, we get
\begin{equation}\label{estimate for limit}
    \begin{split}
    \frac12 
    &\int |\nabla \Sigma_\infty|^2(T) \dx x + \frac\nu8 \int_0^T\!\!\int_{\T^d} |\Delta \Sigma_\infty |^2 \dx x \dx t\\[0.5em]
    &\leq \frac\nu4 \norm{\Delta \tilde V}_{L^2}^2+
    \frac1\nu \norm{\partial_t \tilde V}_{L^2}^2 + \norm{\Delta \tilde V}_{L^2} \norm{\partial_t \tilde V}_{L^2} + C\\[0.3em]
    &\qquad+ \frac1\nu \int_0^T\!\!\int_{\T^d} |\nabla \Sigma_\infty|^2 |\partial_t \tilde V| \dx x \dx t + \frac 1\nu \int_0^T\!\!\int_{\T^d} |\nabla \Sigma_\infty \cdot \nabla \tilde V|^2 \dx x \dx t.
    \end{split}
\end{equation}

\begin{lemma}[$p_\infty, \Sigma_\infty \in L^\infty(0,T;L^\infty(\T^d))$]
The solution $(n_\infty,p_\infty)$ to \eqref{eq: limit} with $V=\tilde V$, satisfies
\begin{equation*}
    \|p_\infty\|_{L^\infty(0,T;L^\infty(\T^d))}\leq 2\|\tilde V\|_{L^\infty(0,T;L^\infty(\T^d))}. 
\end{equation*}
Therefore, we also have $\Sigma_\infty \in L^\infty(0,T;L^\infty(\T^d))$.
\end{lemma}
\begin{proof}
Let $(n_\infty,p_\infty)$ be the unique solution to \eqref{eq: limit}, with  $\Sigma_\infty=\nu n_\infty + p_\infty$ --- for the uniqueness result we refer the reader to \cite{DM2016}.
Let $t_0>0$ and $\epsilon>0$.
We choose a sequence $n_\gamma$ of solutions to \eqref{eq: main-active-motion} for $t\geq t_0$, such that the initial data satisfy $n_\gamma(\cdot,t_0)\leq 1-\frac{1}{\sqrt \gamma}$ for $\gamma$ large enough and $n_\gamma(t_0)$ converges uniformly to $n_\infty(t_0)$ on $\T^d$. Hence, $p_\gamma(t_0)\leq (1-1/\sqrt{\gamma})^\gamma\to 0$ as $\gamma\to \infty$. From the proof of Lemma~\ref{lemma Cs} we know that
\begin{align*}
    \max_x p_\gamma(t) - \max_x p_\gamma(t_0) \leq \max_x \tilde V(t_0)- \min_{x,t} \tilde V + \int_{t_0}^{t}\norm{\nu \Delta \tilde V+\partial_t  \tilde V}_{L^\infty(\T^d)} \dx s.
\end{align*}
Passing to the limit $\gamma\to \infty$, we find
\begin{align*}
 \max_x p_\infty(t) &\leq \max_x \tilde V(t_0)- \min_{x,t} \tilde V + \int_{t_0}^{t}\norm{\nu \Delta \tilde V+\partial_t  \tilde V}_{L^\infty(\T^d)} \dx s\\
 &\leq 2\|\tilde V\|_{L^\infty} +\int_{t_0}^{t}\norm{\nu \Delta \tilde V+\partial_t  \tilde V}_{L^\infty(\T^d)} \dx s.
\end{align*}
Now we aim at choosing $t=t_0+\epsilon$ and passing to $\epsilon\to 0$. Let $(x_0,t_0)$ be a Lebesgue point for $p_\infty$. 
Since the above inequality holds for almost every $t_0\leq t\leq t_0+\epsilon$, we have
\begin{equation*}
\fint_{t_0}^{t_0+\epsilon}\fint_{B(x_0,\epsilon)} p_\infty(x,s) \dx x \dx s \leq  2\|\tilde V\|_{L^\infty} +\int_{t_0}^{t_0+\epsilon}\norm{\nu \Delta \tilde V+\partial_t  \tilde V}_{L^\infty(\T^d)} \dx s,
\end{equation*}
which, as $\epsilon\to 0$, becomes
\begin{equation*}
p_\infty(x_0, t_0)\leq  2\|\tilde V\|_{L^\infty}.
\end{equation*}
This holds for almost every $(x_0,t_0)\in \T^d\times(0,T)$ and thus concludes the proof.
\end{proof}

\begin{remark}
    Let us note that, in the same fashion, also for $\nu=0$ it would be possible to obtain a uniform bound on $p_\infty$ depending only on the $L^\infty$-norm of the potential. However, we are not able to exploit it in the same way as for the case $\nu>0$. In fact, to obtain \eqref{estimate for limit} we pass to the limit $\gamma\to\infty$ and use the lower semi-continuity of the $L^2$-norm of $\Delta \Sigma_\gamma$, which is weakly compact. In the porous medium case without active motion, $\nu = 0$, the same does not hold for the quantity $\|p_\gamma D^2 p_\gamma\|_{L^2}$.
\end{remark}

\smallskip

The above lemma shows that the $L^\infty$-bound of $\Sigma_\infty$ no longer depends on the norm $\|\nu \Delta \tilde V +\partial_t \tilde V\|_{L^1(0,T;L^\infty(\T^d))}$, but only on the $L^\infty$ norm of the potential. A similar estimate, where the same ``instantaneous'' dependence of $p$ in terms of $V$ is shown, was already present in~\cite{LavSanPressure}.

Thanks to Young's inequality, from \eqref{estimate for limit} we compute
\begin{equation*}
    \begin{split}
    \frac12 
    &\int |\nabla \Sigma_\infty|^2(T) \dx x + \frac\nu8 \int_0^T\!\!\int_{\T^d} |\Delta \Sigma_\infty |^2 \dx x \dx t\\[0.5em]
    &\leq \frac\nu4 \norm{\Delta \tilde V}_{L^2}^2+
    \frac1\nu \norm{\partial_t \tilde V}_{L^2}^2 + \norm{\Delta \tilde V}_{L^2} \norm{\partial_t \tilde V}_{L^2} + C\\[0.3em]
    &\qquad+ \varepsilon \int_0^T\!\!\int_{\T^d} |\nabla \Sigma_\infty|^4 \dx x \dx t + C \|\partial_t \tilde V\|^2_{L^2} + \|\nabla \tilde V\|_{L^4}^4\\[0.3em]
    &\leq C +\varepsilon \int_0^T\!\!\int_{\T^d} |\nabla \Sigma_\infty|^4 \dx x \dx t,
    \end{split}
\end{equation*}
where $C$ only depends on the norms $ \norm{\partial_t \tilde V}_{L^2}, \|\Delta \tilde V\|_{L^2}, \|\nabla \tilde V\|_{L^4}$.
Using again that 
\begin{equation*}
    \int_{\T^d} |\nabla \Sigma_\infty|^4 \dx x \leq C \|\Sigma_\infty\|_\infty^2 \int_{\T^d} |\Delta \Sigma_\infty|^2 \dx x,
\end{equation*}
choosing $\varepsilon$ small enough and passing to the limit $\delta\to 0$  we conclude.
\end{proof}

\bigskip

We conclude by proving that also the gradient of the limit pressure, $p_\infty$, is bounded in $L^4$ under the same assumptions, Corollary~\ref{cor: L4 diff}.

\begin{proof}[Proof of Corollary~\ref{cor: L4 diff}]
Let us recall
\begin{equation*}
    \Sigma_\infty= p_\infty + \nu n_\infty \in L^2(0,T; H^2(\R^d)),
\end{equation*}
as well as
\begin{equation*}
    p_\infty(1-n_\infty)=0.
\end{equation*}
As a consequence, we have
\begin{equation*}
    p_\infty=(\Sigma_\infty- \nu)_+,
\end{equation*}
where $(\cdot)_+$ denotes the positive part. Hence, the regularity of $p_\infty$ is, for each $t$, that of the positive part of an $H^2$-function. Moreover, we obtain $p_\infty\in L^4(0,T; W^{1,4}(\R^d))$ under the same assumptions of Theorem~\ref{thm: L4 diff infty}.
\end{proof}

\section*{Acknowledgments}
This project was supported by the LABEX MILYON (ANR-10-LABX-0070) of Université de Lyon, within the program «Investissements d’Avenir» (ANR-11-IDEX-0007) operated by the French National Research Agency (ANR), and by the European Union via the ERC AdG 101054420 EYAWKAJKOS project.
The authors also acknowledge the support of the Lagrange Mathematics and Computation Research Center and the Oberwolfach Research Institute for Mathematics which hosted important discussions on this topic.

\bibliographystyle{abbrv}
\bibliography{literature}

\end{document}